\documentclass[reqno, english]{amsart}
\usepackage{amsfonts}
\usepackage{amsthm}
\usepackage{amsmath}
\usepackage{amscd}
\usepackage[latin2]{inputenc}
\usepackage{t1enc}
\usepackage[mathscr]{eucal}
\usepackage{indentfirst}
\usepackage{graphicx}
\usepackage{graphics}
\usepackage{pict2e}
\usepackage{epic}
\usepackage{bbm}
\usepackage{bm}
\usepackage[left=1in,right=1in,top=1in,bottom=1in]{geometry}
\usepackage{epstopdf} 
\usepackage{enumitem}

\usepackage{amssymb}

\theoremstyle{plain}
\newtheorem{thm}{Theorem}[section]

\newtheorem{lemma}[thm]{Lemma}
\newtheorem{cor}[thm]{Corollary}
\newtheorem{prop}[thm]{Proposition}

\newtheorem{obs}[thm]{Observation}
\newtheorem{conj}[thm]{Conjecture}

\newtheorem*{claim*}{Claim}

\newcommand{\R}{\mathcal R}

\newcommand{\NN}{\mathbb{N}}

\newcommand{\prr}{\pr}
\newcommand{\pr}{\mathbb{P}}

\newcommand{\E}[0]{\mathbb{E}}

\newcommand{\beq}[1]{\begin{equation}\label{#1}}
\newcommand{\enq}[0]{\end{equation}}

\newcommand{\mn}[0]{\medskip\noindent}
\newcommand{\nin}[0]{\noindent}

\newcommand{\sub}[0]{\subseteq}
\newcommand{\sm}[0]{\setminus}

\newcommand{\ov}[0]{\overline}

\renewcommand{\dots}[0]{,\ldots,}

\newcommand{\ra}[0]{\rightarrow}
\newcommand{\Ra}[0]{\Rightarrow}

\newcommand{\0}[0]{\emptyset}

\newcommand{\C}[2]{\binom{{#1}}{{#2}}}
\newcommand{\Cc}[0]{\tbinom}

\newcommand{\gc}[0]{\gamma }
\newcommand{\gd}[0]{\delta }

\newcommand{\go}[0]{\omega}
\newcommand{\gO}[0]{\Omega}

\newcommand{\gs}[0]{\sigma}

\newcommand{\eps}[0]{\varepsilon }
\newcommand{\vt}[0]{\vartheta}
\newcommand{\vs}[0]{\varsigma}

\newcommand{\vp}[0]{\varphi}

\newcommand{\prh}[1][]{\pr_h}

\usepackage{xcolor}

\newcommand{\Var}[0]{{\rm Var}}

\newcommand{\win}[0]{{\rm win}}

\newcommand{\mbN}[0]{\mathbb{N}}

\newcommand{\recovby}[0]{<\hspace{-0.07in}\cdot~}

\newcommand{\Bep}[0]{D}
\newcommand{\Cep}[0]{C}
\newcommand{\Dep}[0]{D}
\newcommand{\Tep}[0]{T}

\title[Variance vs. range, and balancing in posets of bounded width]{Variance vs.\ range for linear extensions, and 
balancing extensions in posets of bounded width.}

\author{Max Aires, Jeff Kahn}

\begin{document}

\maketitle
\begin{abstract}
An old conjecture of Kahn and Saks says, roughly, that any poset $P$ of
large enough width contains elements $x,y$ which are ``balanced'' in the sense that 
the probability that $x$ precedes $y$ in a uniformly random linear extension of $P$
is close to $1/2$. 
We show this implies the seemingly stronger statement that the same conclusion
holds if, instead of large width, we assume only that, for some $x$, the number, $\pi(x)$, of 
elements of $P$ incomparable to $x$ is large.
The implication follows from our two main results:
first, that if $\pi(P):=\max \pi(x)$ is large then $P$ has large variance, i.e.\ 
there is a $y$ whose position in a uniform extension of $P$ has large variance; 
and second, that the conclusion of the Kahn-Saks Conjecture holds for $P$
with large variance and \emph{bounded} width.
These two assertions also yield an easy proof of a (not easy) result of Chan, Pak and Panova
on ``sorting probabilities'' for Young diagrams, together with its natural generalization to higher dimensions.


\end{abstract}

\section{Introduction}

Our main results are Theorems~\ref{pi-var} and \ref{Twsig}, and their
consequences, Theorems~\ref{Cpi} and \ref{TKSpi}.
Before turning to these we fill in some of what led to them,
referring to the companion paper \cite{AK25}
for a more thorough account of this background.
Some basics are recalled at the end of this section.

Following a standard abuse we identify a partially ordered set (poset) $P$ with its ground set.
A \emph{linear extension} of an $n$-element $P$
may then be viewed as either
a linear ordering of the elements of $P$
respecting the poset relations, or as
an order-preserving bijection from $P$ to $ [n]:=\{1\dots n\}$.
We use $E(P)$ for the set of linear extensions of $P$, 
and $\prec$ or $f$
for a uniform member of $E(P)$,
according to whether we are thinking of an ordering or a bijection.
Of central interest here are the ``balancing constants,''
$\gd_{xy}:=\min\{\prr(x\prec y),\prr(y\prec x)\}$, 
$\gd_x:=\max_{y\neq x}\gd_{xy}$,
and
$\gd(P):=\max \gd_{xy}$ (the max over distinct $x,y\in P$).

The primary motivation for \cite{AK25} was a pair of old 
``balancing'' problems for linear extensions,
beginning with the following ``1/3-2/3 Conjecture."

\begin{conj}
\label{C1/3}
If the finite poset P is not a chain then 
$
\gd(P)\geq 1/3.
$
\end{conj}
\nin
\nin
This was proposed by Kislitsyn \cite{Kislitsyn}
and again by Fredman (circa 1975) and Linial \cite{Linial},
all motivated by questions about sorting.
From the algorithmic standpoint what's wanted is existence of
\emph{some} universal $\gd>0$ 
such that $\gd(P)\geq \gd$ whenever $P$ is not a chain.
The first such bound, with $\gd=3/11$, was given in 
\cite{Kahn-Saks},
and the current best value
is $(5-\sqrt{5})/10$, due to
Brightwell, Felsner and Trotter \cite{BFT25}.
(See also \cite{Linial}, \cite{Karzanov-Khachiyan} and \cite{Kahn-Linial}.)

While Conjecture~\ref{C1/3} is tight, 
it seems likely that one can usually do better, as, for example,
in the following natural guess from \cite{Kahn-Saks}, 
now usually called the \emph{Kahn-Saks Conjecture}.
Recall that the \emph{width}, $w(P)$, is the size of a largest antichain in $P$.

\begin{conj}
\label{CKS}
If $w(P) \ra \infty$, then
\beq{1/2}
\gd(P)\ra 1/2.
\enq
\end{conj}
\nin
Here and throughout, assertions like the one in Conjecture~\ref{CKS} have
the natural meanings, e.g.\ in the present case, that for any positive $\eps$,
there is a $w_\eps$ such that $\gd(P)>1/2-\eps$ if $w(P)> w_\eps$.

Prior to \cite{AK25}
(which we will not review here), 
the best general result  on Conjecture~\ref{CKS}
was that of Koml\'os \cite{Komlos}, who showed that \eqref{1/2} holds if 
$\min(P)=\gO(n)$.  
This was based on a lovely probabilistic result that is again crucial for the parts of 
\cite{AK25} involving Conjecture~\ref{CKS}.
See also e.g.\ \cite{Korshunov}, \cite{CPP21} 
for special cases;
Friedman \cite{Fri93} (another antecedent of \cite{AK25}) for conditions implying 
$\gd(P)> 1/e-o(1)$; and
\cite{Chan-Pak} for a good recent account of this area.

Much of \cite{AK25} is concerned with understanding relations between a few
parameters associated with a poset $P$, of which, in addition to the width, we here need only:
the \emph{standard deviation,} 
\[
\mbox{$\gs(P)=\max_{x\in P}\gs(x),$}
\]
where $\gs(x) $ is the standard deviation of $f(x)$ (recall $f$
is uniform from $E(P)$ if we think of bijections),
and the \emph{range}, 
\[
\mbox{$\pi(P)=\max_{x\in P}\pi(x),$}
\]
where $\pi(x) = |\Pi(x)|$, with $\Pi(x) = \{y:y\not\sim x\}$.

As in \cite{AK25}, we use $\gc\leadsto\vt$, for poset parameters $\gc$, $\vt$,
to mean that $\vt$ tends to infinity if $\gc$ does.
For the parameters above, $w\leadsto \pi$ and  $\gs\leadsto \pi$ are trivial, as are the more
concrete $w(P)\leq \pi(P)+1$ and
\beq{gspi}
\pi(x)=\gO(\gs(x)).
\enq
On the other hand, $w\leadsto \gs$, in the stronger form
\beq{gsw}
\gs(P)=\gO(w(P)),
\enq
follows from a key result of \cite{AK25}.
(What's shown in \cite{AK25} is that for another parameter, ``win,''
any antichain $A$ satisfies $\sum_{x\in A}\win(x)=\gO(|A|^2)$;
which, since $\gs(x)=\gO(\win(x))$ is easy, implies a quantified and
``\emph{local}'' version of \eqref{gsw}:
any antichain $A$ contains $x$ with $\gs(x) =\gO(|A|)$.)


Neither $w\leadsto \gs$ nor \eqref{gspi} is reversible;
e.g.\ we can
have $w(P)=2$ and $\gs(P)= \gO(n)$ 
(let $P$ be a chain plus an isolated point), or $\pi(x)=\gO(n)$ while both
$w$ and $\gs(x)$ are $O(1)$ 
(as for the disjoint union of two chains of equal size with $x$ minimal in one of them). 
The main result of the present work says that
a (qualitative) \emph{global} converse of \eqref{gspi} does hold:

\begin{thm}\label{pi-var} 
$\,\,\pi \leadsto \gs$.

\end{thm}
One motivating idea for \cite{AK25} was that each of the 
the parameters discussed there could
replace width in Conjecture~\ref{CKS}, in 
the sense that \eqref{1/2} should hold if any one of them tends to infinity.
The strongest of these possibilities
is an old conjecture of the second author:
\begin{conj}\label{maxC}
If $~\pi(P)\ra\infty$ then \eqref{1/2} holds.
\end{conj}
\nin
Since $w\leadsto \gs\leadsto \pi$,
Conjecture~\ref{maxC} implies the same statement with $\gs$ in place of $\pi$, which in turn implies
Conjecture~\ref{CKS}; similarly Conjecture~\ref{maxC} implies the other 
possibilities from \cite{AK25} alluded to above.

Theorem~\ref{pi-var} and
Conjecture~\ref{maxC} were shown 
in \cite{Aires}
for $P$ of width 2, and we will show the former here for any fixed width
(see Theorem~\ref{Cpi}).
The present argument is different from, and considerably
easier than, that of \cite{Aires}, though still not very easy.
The main points are Theorem~\ref{pi-var} (this is the hard part) and
the next statement, which is our second main result.

\begin{thm}\label{Twsig}
For fixed $w$, $P$ of width at most $w$, and $ x\in P$,
\[
\mbox{$\gd_x\ra 1/2$ as $\gs(x)\ra\infty$.}
\]  
In particular, \eqref{1/2} holds if $\gs(P)\ra\infty$.
\end{thm}
\nin
In view of Theorem~\ref{pi-var} this gives the above-mentioned 
bounded width version of Conjecture~\ref{maxC}.
\begin{thm}\label{Cpi}
For fixed $w$ and $P$ of width at most $w$, 
if $\pi(P)\ra\infty$ then \eqref{1/2} holds.
\end{thm}
\nin
Theorems~\ref{pi-var} and \ref{Twsig} also yield
the surprising fact
that Conjecture~\ref{CKS} is equivalent to the
seemingly less plausible (since its hypothesis is so weak) Conjecture~\ref{maxC}:

\begin{thm}\label{TKSpi}
Conjecture~\ref{CKS} implies Conjecture~\ref{maxC}.
\end{thm}

\nin
To see this, just note that Conjecture~\ref{CKS} implies its ``$\gs$ version'' 
(which according to Theorem~\ref{pi-var} implies Conjecture~\ref{maxC}), 
since Theorem~\ref{Twsig} reduces proving the $\gs$ version to 
proving it for large width.
(Precisely:  for any $\eps>0$, Conjecture~\ref{CKS} gives $\gd(P)> 1/2-\eps$ if
$w(P)>w_\eps$, and Theorem~\ref{Twsig} gives the same conclusion for $P$ with $w(P)\le w_\eps$
and $\gs(P)$ sufficiently large.)

Finally, we have the following consequence of Theorem~\ref{Cpi}, 
which greatly extends the motivating result of \cite{CPP21}.
(We use $\mbN$ for the set of positive integers, endowed with the product order.
An \emph{ideal} 
of $P$ is $I\sub P $ satisfying $x<y\in I\Ra x\in I$, and $Q\sub P$ is \emph{convex}
if $[\mbox{$x,z\in Q$ and $x<y<Z$}]$ $\Ra$ $y\in Q$.)

\begin{thm}\label{TYd}
For fixed $d$ and $w$, and $P$ ranging over convex subsets 
of $\NN^d$ with $w(P)\in \{2\dots w\}$,
\[
\mbox{$\gd(P)\ra 1/2$ as $|P|\ra\infty$.}
\]  
\end{thm}
\nin
For $d=2$, an ideal of $\mbN^d$ is a \emph{Young diagram}
and a convex $P\sub \mbN^d$ is a \emph{skew Young diagram};
for these cases
Theorem~\ref{TYd} was shown by Chan, Pak and Panova~\cite{CPP21}.
(Strictly speaking, they assume a bound on the number of rows of the diagram, which is 
a little stronger than bounded width.)

That Theorem~\ref{Cpi} implies Theorem~\ref{TYd} follows from the next observation.

\begin{thm}\label{down} For all $d>0$:

\nin
{\rm (a)}
any convex $P\sub  \NN^d$ that is not a chain has
$
\pi(P)\geq |P|/2-1.
$

\nin
{\rm (b)}
 any ideal $P$ of $ \NN^d$ 
not contained in a coordinate hyperplane has
$   
\pi(P)> (1-1/d)|P|+O_d(1).
$   
\end{thm}
\nin
(We just need $\pi(P)=\gO(|P|)$ but think it worth recording the more precise bounds.)

\mn
\textbf{Outline:}
Section~\ref{VandB} recalls fundamental results of Shepp and Stanley, and proves 
Theorem~\ref{Twsig}; 
Section~\ref{IPAV} reduces Theorem~\ref{pi-var} to a special case (Theorem~\ref{pi-var-gen}),
the proof of which, given in Section~\ref{SecVW2P}, was the most demanding
 part of the present work; and the brief Section~\ref{SecYd} proves Theorem~\ref{down}.

\mn
\textbf{Usage}

As above we use ``$<$'' for order in $P$ and ``$\prec$'' 
for order in a uniform extension, 
in each case adding a dot for \emph{cover} relations (e.g.\ $x\recovby y$
means $x\leq w\leq y$ $\Ra $ $w\in \{x,y\}$).
As usual $X<Y$ means $x<y$ $\forall (x,y)\in X\times Y$.
Elements $x,y$ of $P$ are \emph{comparable}, denoted $x\sim y$, if 
$x<y$ or $x>y$.

For $Y$ contained in (the ground set of) $P$, the \emph{subposet} $P[Y]$
is the poset with ground set $Y$ and order inherited from $P$.
We use $P=X\sqcup Y$ to mean the \emph{set} $P$
is the disjoint union of $X$ and $Y$, and $P=X+Y$ if, in addition, $X\not\sim Y$.
A \emph{chain} is a linearly ordered set and
an \emph{antichain} is a set with no relations $x<y$.
As above, the \emph{width}, $w(P)$, is the maximum size of an antichain.

The sets of maximal and minimal elements of $X\sub P$ are denoted $\max (X)$ and $\min (X)$.
An (\emph{order}) \emph{ideal} (nowadays also called a \emph{downset}) of $P$ is $I\sub P $ satisfying
$x<y\in I\Ra x\in I$, and \emph{filter} (\emph{upset}) is defined similarly.

As above, $E(P)$ is the set of linear extensions (we will sometimes say simply \emph{extensions})
of $P$.  We use either $\prec$
or $f$ for a uniform member of $E(P)$ (depending on viewpoint),
and $\pr$ and $\E$ for 
the associated probabilities and expectations.  (We will also, as necessary,
use $\pr_Q$ and $\E_Q$ with $Q$ an auxiliary poset.)

As usual $[n]=\{1\dots n\}$ and events are designated with $\{\cdots\}$.
We use standard asymptotic notation ($O(\cdot)$, $\gO(\cdot)$ and so on),
and usually  
pretend large numbers are integers, rather than clutter the discussion with
irrelevant floor and ceiling symbols.

\section{Variance and Balance}\label{VandB}

Before turning to Theorem~\ref{Twsig}, we recall a pair of basic results, the first of which is Shepp's
``XYZ Inequality'' \cite{She82}:
\begin{thm}\label{TXYZ}
For $x\in P$ and $Y\sub P- x$,
\[
\prr(x\succ Y)\geq \prod_{y\in Y}\prr(x\succ y).
\]
\end{thm}
\nin
(The statement in \cite{She82} is for $|Y|=2$, but implies Theorem~\ref{TXYZ} 
via a straightforward induction.)

The next result, due to Stanley \cite{Sta81},
is used here and again in the proof of Theorem~\ref{pi-var} (see Section~\ref{SecVW2P}).
\begin{thm}\label{Stanley} 
For any P and x, the sequence $\pr(f(x)=k)$ is log-concave.
\end{thm}
\nin
(That is,  $\pr(f(x)=k)^2\ge \pr(f(x)=k-1)\pr(f(x)=k+1)$ for all $k$.)

\nin
We use this with
the following standard observation
(for which we unfortunately don't know a reference).

\begin{prop}\label{Ldisclc}
Any 
integer-valued random variable $f$ with 
a log-concave distribution has
\[
\Var(f) =O((\max_k\pr(f(x)=k))^{-2}).
\]
\end{prop}

With 
$      
q(x):= \max_k \pr(f(x)=k)$, we have the following important consequence.
\begin{cor}\label{Cgsq}
For any $P$ and $x$, $\gs(x)=\Theta(1/q(x))$.
\end{cor}

\nin
\emph{Proof.}
That $\gs(x)=\gO(1/q(x))$ is obvious (even without log-concavity), and $\gs(x)=O(1/q(x))$
is given by Theorem~\ref{Stanley} and Proposition~\ref{Ldisclc}.
\qed

In view of Corollary~\ref{Cgsq} we may rewrite Theorem~\ref{Twsig} as

\begin{thm}\label{Twsig'}
For fixed $w$, $P$ of width at most $w$, and $ x\in P$,
$
\mbox{$\gd_x\ra 1/2$ as $q(x)\ra 0$.}
$  
\end{thm}

\nin
\emph{Remark.}
When $w(P)=2$, Theorem~\ref{Twsig'} is easy:
Let $a=|\{y:y\le x\}|$ and $\Pi(x)=\{y_1<y_2<\cdots<y_r\}$.
Since each of $\prr(x\prec y_1)$ ($=\pr(f(x)=a)$ and $\prr(x\succ y_r)$ ($=\pr(f(x)=a+r)$)
is at most $q(x)$, we have
$
k:=\max\{i: \prr(x\succ y_i)>1/2\}\in [r-1],
$
and
\[
1/2 < \prr(x\succ y_k)= \prr(x\succ y_{k+1}) + \prr(f(x)=a+k) \leq 1/2+q(x).
\]

For general (fixed) $w$, the main point for the proof of Theorem~\ref{Twsig'}
 is the following statement, which we think is interesting in its own right.

\begin{thm}\label{Cwsig}
Suppose $P$ is the disjoint union of $\{x\}$, $D$ and $U$, with $D$ an ideal and $U$ a filter;
that $A:=\max(C)$ and $B:=\min(U)$ satisfy $|A|,|B|\leq w$; and that
\beq{paxb}
\mbox{$\prr(a\prec x\prec b) \geq \eps \,$ for all $\, a\in A $ and $b\in B$.}
\enq
Then $\prr(A\prec x\prec B) \geq \eps^{w^2}$.
\end{thm}
For Theorem~\ref{Twsig'} (or \ref{Twsig}), notice that the assumptions of 
Theorem~\ref{Cwsig} hold if $w(P)\le w$ and $\gd_x\le 1/2-\eps$ 
(with $D=\{y:\prr(x\succ y)\geq 1/2+\eps\}$, $U=\{y:\prr(x\prec y)\geq 1/2+\eps\}$, and
$2\eps$ in place of $\eps$ in \eqref{paxb}).
Thus $\prr(f(x) = |A|+1) \geq (2\eps)^{w^2}$, which is $\gO(1)$ if $w$ and $\eps>0$ are fixed.
(So if $w(P)=O(1)$ and $q(x) =o(1)$, then $\eps$ must be $o(1)$.)


\nin
\emph{Proof of Theorem~\ref{Cwsig}.}
According to \eqref{paxb}
(with $a$ and $b$ always in $A$ and $ B$ respectively),
\[
\prr(x\prec b|a\prec x) \geq \eps/\prr(a\prec x) \,\,\,\forall a,b,
\]
and Theorem~\ref{TXYZ} gives, for any $a$,
\[
\prr(x\prec B|a\prec x)\geq \left[\eps/\prr(a\prec x)\right]^{|B|}
\geq \left[\eps/\prr(a\prec x)\right]^{w},
\]
whence $\prr(a\prec x\prec B)\geq \eps^w$. 
So 
$\prr(a\prec x|x\prec B)\geq \eps^w/\prr(x\prec B)$, and a second application of 
Theorem~\ref{TXYZ} gives 
$\prr(A\prec x|x\prec B)\geq \left[\eps^w/\prr(x\prec B)\right]^w$ and
$\prr(A\prec x\prec B)\geq \eps^{w^2}$.

\qed
 
\nin

\section{Incomparable Pairs and Average Variance}\label{IPAV}

We now turn to Theorem~\ref{pi-var}. Since the theorem is not true ``locally'' 
(that is, $\pi(x)$ large doesn't imply $\gs(x)$ large; see the paragraph preceding the 
statement of the theorem),
a basic difficulty is understanding where, 
given only that some $\pi(x)$ is large, 
one should look to find elements of $P$ with 
large $\gs$.  

For disjoint $A,B\sub P$, call the pair $(A,B)$ \emph{incomparable} if $a\nsim b$ for all 
$(a,b)\in A\times B$,
and call an incomparable $(A,B)$
\emph{$\mu$-near maximum}---or a \emph{$\mu$-pair}---if $|A||B|\ge \mu |C||D|$ for 
every incomparable $(C,D)$. 
Theorem~\ref{pi-var} follows from the next assertion.

\begin{thm}\label{pi-var-gen} 
For all $\mu>0$ and $L$ there is a $K$ such that 
if $(A,B)$ is a maximal
$\mu$-pair with $|B|\ge \max\{\mu|A|, K\}$, then 
\beq{sumxA}
\sum_{x\in A}\gs^2(x)\ge L|A|.
\enq
\end{thm}
\nin
Thus the conditions on $(A,B)$ imply that even the \emph{average} variance
of elements of $A$ is large.
In the situation of Theorem~\ref{pi-var}, we may 
of course take $\mu=1$:  choose $(A,B)$ incomparable with 
$|A||B|$ maximum (so $|A||B|\geq 1\cdot \pi(P) =\go(1)$) and $|B|\geq |A|$.

In the rest of our discussion we will usually be working 
with $P$ satisfying
\beq{Phyps}
\mbox{$P=X\sqcup Y$ with $X=\{x_1<\cdots<x_m\}$ and 
$Y=\{y_1<\cdots<y_n\}$.}
\enq
We call such a $P$ \emph{unidirectional} if it has no relations $x_i>y_j$.

We next derive Theorem \ref{pi-var-gen} from the following variant, which is proved in 
the Section~\ref{SecVW2P}.

\begin{thm}\label{main-lemma} 
For all $\mu>0$ and $L$ there is a $K$ such that if
 $P$ is as in \eqref{Phyps} and unidirectional,
and $(A,B)$ is a $\mu$-pair with $A\sub X$, $B\sub Y$
and $|B|\ge \max\{\mu|A|, K\}$, then 
\[    
\sum_{x\in A}\gs^2(x)\ge L|A|.
\]    
\end{thm}

\begin{proof}[Proof of Theorem~\ref{pi-var-gen} assuming Theorem~\ref{main-lemma}]
Let $K$ be as in Theorem~\ref{main-lemma} with 
$2L$ in place of $L$ and $\mu/4$ in place of $\mu$,
and let $A$, $B$ be as in Theorem~\ref{pi-var-gen}.
Set
\[
C=\{x\in P\sm(A\cup B): x<z\text{ for some }z\in A\cup B\}
\] 
and
\[
D=P\sm(A\cup B\cup C),
\]
and notice that 
\beq{ifuv}
\mbox{if $u<v$ and $u,v$ are in different members of $\{A,B,C,D\}$, then $u\in C$ or $v\in D$.}
\enq
\emph{Proof.}  
We cannot have $u\in D$ since this would forbid $u< v\in A\cup B\cup C$.
If $u\in A$ then $v\not \in B$ (since 
$(A,B)$ is incomparable); and if $v\in C$, say $v<w\in A\cup B$,
then $w\in A$ (since $w\sim u\in A$) and $v\not\sim B$ (or one of $u,w$ 
is related to $B$); so $(A\cup \{v\},B)$ is incomparable, contradicting maximality of $(A,B)$.
Of course the same discussion holds with $B$ in place of $A$, so $u\not\in C$ implies 
$v\in D$ as desired. \qed

We next observe that it is enough to show \eqref{sumxA} when $A,B,C,D$ are chains,
since (i) $(A,B)$ is a $\mu$-pair in any $Q$ obtained by adding relations to $P$ without 
adding any relations between $A$ and $B$, and (ii) with $\vs$ encoding the restrictions of $\prec$
to $A,B,C$ and $D$,
\[
(\gs^2(x)=)\,\,\,
\Var(f(x)) = \E[\Var(f(x)|\vs)] + \Var(\E[f(x)|\vs])\geq  \E[\Var(f(x)|\vs)].
\]

Let $a$ and $b$ be the centers of $A$ and $B$, 
and assume for simplicity that $|A| $ and $|B|$ are odd.
Since the statement in Theorem~\ref{pi-var-gen} doesn't change if we replace $P$ by its
``dual'' (obtained by reversing the order relations of $P$), we may assume 
that $\prr(a\succ b)\ge 1/2$; so with $A'=\{x\in A:x\geq a\}$ and 
$B'=\{y\in B:y\le b\}$, we have $\prr(A'\succ B')\ge 1/2$.

Let $Q$ be the (random) poset obtained by specifying the restrictions of $\prec$ to 
$C\cup A$ and $B\cup D$ (and closing transitively), 
and let $\R$ be the event that $a\not\sim b$ in $Q$.
Then $C\cup A$ and $B\cup D$
are chains of the width 2 poset $Q$, and there are no relations $x>y$ with 
$x\in C\cup A$ and $y\in B\cup D$ (since \eqref{ifuv} says this is true in $P$).
In particular, $\prr(\R)\ge \prr(a\succ b)\ge 1/2$.

But if $\R$ holds, then $(A',B')$ is incomparable in $Q$, and thus a $(\mu/4)$-pair (since 
$|A|'|B'|\ge |A||B|/4$) with $|B'|\ge \mu|A'|$;
so Theorem \ref{main-lemma}, with our choice of $K$, gives
\[
\sum_{x\in A}\Var(f(x))\ge
\prr(\R)\sum_{x\in A'} \Var(f(x)\mid \R) \ge 2L |A'|= L|A|.
\]
\end{proof}

\section{Variance in Width-2 Posets}\label{SecVW2P}

It remains to prove Theorem \ref{main-lemma}, which we will do here
following a little more preparation.

The main tools underlying our argument are Stanley's Theorem~\ref{Stanley}
(used via Corollary~\ref{Cgsq}) and the 
following classic result of Graham, Yao and Yao \cite{GYY80}.

\begin{thm}\label{GYY} 
For $P$ as in \eqref{Phyps}, $x\in X$, $y\in Y$, $I\sub X\times Y$,
and $\R=\{x'\prec y' \,\,\forall (x',y')\in I\}$,
\[
\prr(x\prec y\mid \R)\,\ge \,\prr(x\prec y).
\]
\end{thm}

\nin
\emph{Remark.}
The original proof of Theorem~\ref{GYY} 
(and its precursor in \cite{CFG})
viewed linear extensions of $P$ as in \eqref{Phyps} as walks on the lattice of ideals of $P$, 
members of which may be identified 
with elements of $\{0\dots m\}\times \{0\dots n\}$; 
so
$(i,j)$ corresponds to the ideal generated by $\{x_i,y_j\}$,
and walks go from $(0,0)$ to $(m,n)$
via increments $(1,0)$ and $(0,1)$, where (e.g.) a step
$(i,\cdot)\ra (i+1,\cdot)$ corresponds to 
adding $x_{i+1}$ to an extension that we are specifying by adding elements of $P$ one at a time.
The same viewpoint is taken
by Chan, Pak, and Panova \cite{CPP22} in their proof of the \emph{Cross Product Conjecture} 
for width $2$ posets,
and in the proofs of Theorem~\ref{pi-var} and
Conjecture~\ref{maxC} for width 2 in \cite{Aires},
and was helpful in discovering the present argument.
We will not use this formulation---or attempt translations---here, but a reader may (or may not)
find the perspective helpful in visualizing what's happening below.

\emph{For the rest of our discussion $P$ is as in \eqref{Phyps}}
(but need not satisfy the conditions of Theorem~\ref{main-lemma}
until we come to the theorem's proof at the end of the section.)
With $x$'s and $y$'s always taken to lie in $X$ and $Y$ respectively, we define r.v.s
\beq{g(x)}
g(x) = |\{y: y\prec x\}|.
\enq
(So $g(x_i)=f(x_i)-i$.)  
We will find one small use for an important geometric result of Gr\"unbaum~\cite{Grunbaum},
which in the present setting implies that (for any $x$)
 \beq{grun}
\mbox{each of $\pr(g(x)\geq \E g(x))$, $\pr(g(x)\le \E g(x))$ is at least $1/e$.}
\enq
(More generally, $\pr(f(x)\geq h(x))> 1/e$ for \emph{any} $P$ and $x$; see 
\cite{Karzanov-Khachiyan,Kahn-Linial,Fri93,AK25} for more substantial uses of 
Gr\"unbaum's theorem and variants.)

The case $P=X+Y$ 
will be important here,
and for this we will need the easily checked
\beq{binom'}
|E(P)| =\C{m+n}{m}
\enq
and
\beq{Egxi}
\E g(x_i) = i\cdot n/(m+1).
\enq

\emph{From this point we assume $P$ is (as in \eqref{Phyps} and) unidirectional,
and use $i$ and $j$ 
for indices from $[m]$ and $[n]$ respectively.}
We will mainly be interested in the events
\[
\Psi_{i,j}=\{y_j\prec x_i\prec y_{j+1}\},
\]

\nin
and will also make use of
$\Phi_{j,i}=\{x_i\prec y_j\prec x_{i+1}\}$
(with $\Psi_{i,0}:=\{x_i\prec y_1\}$ and $\Phi_{j,m}:=\{y_j\succ x_m\}$).
The effect of conditioning on such events---which we will do repeatedly---is easily understood:

\begin{obs}\label{PQobs}
For $i,j\in [m]\times [n]$, 
$Q=P[\{x_1\dots x_{i-1}\}\cup \{y_1\dots y_{j}\}]$,
and $E$ any event involving only elements of $Q$,
\[
\pr(E\mid \Psi_{i,j}) =\pr_Q(E)
\]
\end{obs}
\nin
(since the conditioning puts $x_i$ above everything in $Q$ and below everything else).
Analogous statements apply to conditioning on $\Psi_{i,j}\wedge \Psi_{k,\ell}$
or $\Psi_{i,j}\wedge \Phi_{\ell,k}$.

\begin{lemma}\label{BL1}
{\rm (a)}If $i< \eps j$ then $\pr(\Psi_{i,j})<\eps$.

\nin
{\rm (b)}
If $P=X+Y$ and $m<\eps n-1$, then $\pr(\Psi_{i,j})<\eps$
for any $i$ and $j$.
\end{lemma}
\nin
\emph{Proof.} (a)  It is enough to show this conditioned on
$\{x_i\prec y_{j+1}\}\wedge\{y_j\prec x_{i+1}\}$,
a prerequisite for $\Psi_{i,j}$.
By Theorem~\ref{GYY} (and unidirectionality),
the probability is then maximized
when there are no other relations between $X$ and $Y$, 
in which case it is equal to 
\[
\Cc{i-1+j}{i-1}/\Cc{i+j}{i}=i/(i+j)  < \eps.
\]
(b)  For any $i,j$, either $i<\eps j$ or $m-i+1<\eps (n-j)$; so
the assertion is given by (a) applied to either $P$
or the dual poset $P^*$.  (Note $P^*$ is still unidirectional since $P=X+Y$.)

\qed

\begin{lemma}\label{BL2} 
For any $\eps>0$ there is a $K$ such that, if 
$j, n-j> K$ and 
$\frac{n-j}{m-i}<(1+\frac{1}{K})\frac{j}{i}$, then
\[
\pr(\Psi_{i,j})< \eps.
\]
\end{lemma}
\nin
\emph{Proof.}
We may assume 
$x_i<y_{j+1}$ (since this is a prerequisite for $\Psi_{i,j}$) and then,
by Theorem~\ref{GYY}, that there are no other relations between 
$X$ and $Y$.
The proof is then just a calculation to show (roughly) that for $\ell< j$ with
$j-\ell$ small relative to $j$, $\pr(\Psi_{i,\ell})$ is not much less than 
$\pr(\Psi_{i,j})$.  This yields many $\Psi_{i,\ell}$'s with
probabilities not much less than $\pr(\Psi_{i,j})$, which gives the lemma.
Concretely:

For $\ell\le j$, four applications of \eqref{binom'} give
\beq{Pi2j2'}
\frac{\pr(\Psi_{i,\ell-1})}{\pr(\Psi_{i,\ell})}
=\frac{\C{i+\ell-2}{\ell-1}\C{m-i+n-\ell+1}{n-\ell+1}}{\C{i+\ell-1}{\ell}\C{m-i+n-\ell}{n-\ell}}
=\frac{1+(m-i)/(n-\ell+1)}{1+(i-1)/\ell}.
\enq
Under our hypotheses (with a large enough $K$),
we may then choose $B>e/\eps$ so that for $j\geq \ell> j-B$, the r.h.s.\ of \eqref{Pi2j2'}
is at least $1-1/(B+1)$; so $|\{\ell: \pr(\Psi_{i,\ell}) > \pr(\Psi_{i,j})/e\}| \geq B$,
and the lemma follows.

\qed

\begin{proof}[Proof of Theorem~\ref{main-lemma}]
We now assume $P,A,B$ are as in Theorem~\ref{main-lemma},
with $\min(A)=\{x_{m-M+1}\}$ and $\max(B)=\{y_N\}$.
Then the pair $(\{x_{m-M+1},\dots,x_m\}, \{y_1,\dots,y_N\})$ is incomparable
(since $P$ is unidirectional),
 so $|A||B|\ge \mu MN$, implying $|A|\ge \mu M$, $|B|\geq \mu N$, and 
 $N\ge |B|\ge \mu |A|\geq \mu^2 M$.

Given $L$, choose $\eps$ so that $q(x)\le \eps$ implies $\gs^2(x)\ge 2L$
(see Corollary~\ref{Cgsq}). 
It is then enough to show---and the rest of our discussion will 
be devoted to this---that, for large enough $K$,
\beq{xAtoshow}
|\{x\in A:q(x)\le\eps\}|\geq |A|/2.  
\enq
\emph{Proof of \eqref{xAtoshow}.}
Set (recall we pretend large numbers are integers)
\[
\ell = \max\{m-5M/\mu,0\}. 
\] 
(If $\ell=0$ we may add  $x_0< \{x_1,y_1\}$ to $P$, without affecting
probabilities of interest).  Then
\beq{eitheror}
x_{\ell}<y_{ N/4}.
\enq
(If not, then
$(\{x_{\ell+1},\dots,x_m\},\{y_1,\dots,y_{ N/4}\})$ is incomparable 
with $(m-\ell)N/4= (5M/\mu) N/4> MN/\mu\ge |A||B|/\mu$, a contradiction.)

Suppose first that $N\ge 20M/(\eps\mu)$.
Here it will turn out that $q(x)<\eps$ for \emph{every}
$x\in A$; thus we show, for any $x_i\in A$ and $j\in [n]$,
\beq{prXij}
\pr(\Psi_{i,j})< \eps.
\enq

If $j\le N/2$, then we claim that \eqref{prXij} is even true ``locally'':
\beq{prPsi}
\pr(\Psi_{i,j}\mid \Psi_{m-M, j'}\wedge \Phi_{N,i'}) <\eps 
\,\,\,\,\,\mbox{for all feasible $j'\in [n]$ and $i'\in [m]$}
\enq
(where ``feasible'' means $\Psi_{m-M, j'}\wedge \Phi_{N,i'}\neq\0$, 
and if $n-M=0$ we may take $\pr(\Psi_{m-M, j'})=1$).

\nin
\emph{Proof.}  We may assume $j'\leq j$ and $i'\geq i$ (or the probability is zero).
The probability in \eqref{prPsi} is then equal to $\pr_Q(\Psi_{i,j})$, where 
$Q= \{x_{n-M+1}<\cdots < x_{i'}\}+\{y_{j'+1}<\cdots <y_{N-1}\}$
(cf.\ Observation~\ref{PQobs});
and the chain sizes $i'-(n-M)$ and $N-1-j'$
are as in Lemma~\ref{BL1}(b), since $N\ge 20M/(\eps\mu)$ and $j\le N/2$.\qed

\nin

If instead $j>N/2$, then \eqref{eitheror} says it is enough to show that
$\pr(\Psi_{i,j}|\Psi_{\ell,j'})< \eps$ for any $j'\le  N/4$,
which follows from 
Observation~\ref{PQobs} and Lemma~\ref{BL1}(a)
since
\[
i-\ell \le 5M/\mu < \eps N/4 < \eps (j-j').
\]

We turn to the more interesting case $N< 20M/(\eps\mu)$.
This will involve constants $\Bep$ and $\Cep$, which, along with the 
constants implicit in $O(\cdot)$ and $\gO(\cdot)$, are functions of $\mu$ and $\eps$.
It will be clear below that we can choose $\Bep$ and $\Cep/\Bep$ large 
enough to support what we do; roughly, we need
$\Bep$ large enough for \eqref{pryjB} and $\Tep > 4\Bep K$, where 
$T\approx \mu^3\Cep $ is as in \eqref{i2} and $K$ is from 
Lemma~\ref{BL2} (see the proof of \eqref{i2}).
We then assume $M$ and $N$ are much larger, as is true 
if $|A|$ in \eqref{xAtoshow} is large.

Set
\[
I=\{i\in [m-M+1,m]: q(x_i)> \eps\}.
\]
We will show that $|I|=O(1) $, so
\eqref{xAtoshow} holds for large enough $A$.
For each $i\in I$ there is a $j=j(i) \in \{0\dots n\}$ 
with $\pr(\Psi_{i,j})> \eps$, 
whence,
by Corollary~\ref{Cgsq}, 
\beq{pryjB}
\prr(y_{j-\Bep}\prec x_i\prec y_{j+\Bep})>  1-\eps/4
\enq
for large enough $\Bep$.  

Let $\{i_1<\cdots<i_t\}\sub I$ be maximal subject to $i_{k+1}-i_k \geq \Cep$ (for $k\in [t-1]$), 
and set $j(i_k)=j_k$. 
We will show that
\beq{tbd}
t=O(1),
\enq
whence $|I|< 2\Cep t=O(1)$, as promised.

For $k\in [t-1]$, let $r_k=(j_{k+1}-j_k)/(i_{k+1}-i_k)$.
For \eqref{tbd} it is enough to show:
\beq{i1} 
r_k= \gO(1)
\,\,\,\forall k\in [t-1];
\enq
\beq{i2}
r_{k+1}> (1+\Bep/\Tep) r_k \,\,\,\forall k\in [t-2],
\enq
with $\Tep$ equal to $\Cep$ times the lower bound in \eqref{i1}; and
\beq{i3}
r_k< 3/\eps  \,\,\,\forall k\in [t-1];
\enq
(Of course \eqref{i1} follows from its truth for $k=1$ plus \eqref{i2},
but the proof of \eqref{i2} uses \eqref{i1}.)

\mn

We now use
$a\pm b$ to denote a quantity in $(a-b,a+b)$.

\nin
\emph{Proof of \eqref{i1}.} Here we work with $g(x)$ (see \eqref{g(x)}), 
noting to begin that, for any $k\in [t]$, 
\[
j_k=\E g(x_{i_k}) \pm \Bep.
\]
(This follows from \eqref{pryjB} (used with $j=j_k$) and \eqref{grun},
which give (e.g.)
\[
\pr(j_k+\Bep > g(x_k) > \E g(x_k)) > 1/e-\eps/2>0.)
\]

It is thus enough to show that
\beq{jj}
\E [g(x_{i_{k+1}})-g(x_{i_k})] =\gO(i_{k+1}-i_k),
\enq
since, with $\gc$ the implied constant in \eqref{jj}, we then have
\eqref{i1} if $\gc \Cep$ is (slightly) large relative to $\Bep$.
(The $\gc$ produced below, roughly $\mu^3$, doesn't involve $\eps$; 
but the stronger requirement---see the proof of \eqref{i2}---that $\Tep/\Bep$ be larger 
than the $K$ of Lemma~\ref{BL2} does.)

\nin
\emph{Note}:
The argument below will twice use Observation~\ref{PQobs} to pass from $P$ 
with some conditioning to a related $Q$; here we take $g(x)$ to
continue to count elements of $Y\sm Q$ that the conditioning
puts below $x$.

\nin
\emph{Proof of \eqref{jj}}.
If $j_{k+1}\ge N/2$, then,
for any $j<N/4$ and $j'=j_{k+1}\pm \Bep$,
\begin{eqnarray}
\E [j'-g(x_{i_k})\mid \Psi_{i_{k+1},j'}\wedge \Psi_{\ell,j}] & \ge &
j'-j -(i_k-\ell)(j'-j)/(i_{k+1}-\ell)\label{js1}\\
&=& (j'-j)\cdot  (i_{k+1}-i_k)/(i_{k+1}-\ell) \,=\,  \gO (i_{k+1}-i_k).
\label{js2}
\end{eqnarray}
[\emph{Because}: 
With $Q=\{x_{\ell+1}<\cdots <x_{i_{k+1}-1}\}+\{y_{j+1}<\cdots <y_{j'}\}$,
Observation~\ref{PQobs} and Theorem~\ref{GYY} lower bound
the l.h.s.\ of \eqref{js1} by $\E_Q[(j'-j)-(g(x_{i_k})-j)]$,
which by \eqref{Egxi} is equal to the r.h.s.; and the 
l.h.s.\ of \eqref{js2} is $\gO (i_{k+1}-i_k)$ because $j'-j> N/4-\Bep$, $i_{k+1}-\ell < 5M/\mu$, and 
$N > \mu^2M$.]

It follows that
$\E[g(x_{i_{k+1}})-g(x_{i_k})\mid g(x_{i_{k+1}})=j_{k+1}\pm \Bep]
=\gO(i_{k+1}-i_k)$
(since $\pr(\cup_{j<N/4}\Psi_{\ell,j}) =1$), and, 
in view of \eqref{pryjB}, that
\[
\E[g(x_{i_{k+1}})-g(x_{i_k})]\geq \pr(g(x_{i_{k+1}})=j_{k+1}\pm \Bep)\,
\E[g(x_{i_{k+1}})-g(x_{i_k})\mid g(x_{i_{k+1}})=j_{k+1}\pm \Bep]
=\gO(i_{k+1}-i_k);
\]
so we have \eqref{jj}
in this case.
\mn

Now assume $j_{k+1}\le N/2$, so also $j_k\leq N/2$ (which is all we now use).
Then for any $j=j_k\pm \Bep$, 
\[
\E [g(x_{i_{k+1}})-j\mid \Psi_{i_k,j}]\ge
\E [g(x_{i_{k+1}})-j\mid \Psi_{i_k,j}\wedge \Phi_{N,m}\}]
\]
(where, since $\Phi_{N,m}=\{y_N\succ m\}$, the inequality is given by Theorem~\ref{GYY}).
But by Observation~\ref{PQobs}, the second expectation is the same as 
$\E_Q [g(x_{i_{k+1}})-j]$, where 
$Q=P[\{x_{i_k+1}\dots x_m\}\cup \{y_{j+1}\dots y_{N-1}\}]
= \{x_{i_k+1}<\cdots < x_m\}+\{y_{j+1}<\cdots <y_{N-1}\}$, which expectation is,
according to \eqref{Egxi}, equal to 
\[
(i_{k+1}-i_k)(N-j)/(m-i_k+1) =\gO(i_{k+1}-i_k);
\]
and combining with \eqref{pryjB} gives \eqref{jj}.

\mn
\emph{Proof of \eqref{i2}.} 
If $k$ violates \eqref{i2},
then for
$j=j_k\pm \Bep$ and $j' =j_{k+2}\pm \Bep$,
\begin{eqnarray}
(j'-j_{k+1})/(i_{k+2}-i_{k+1})&<&
(1+\Dep/\Tep)(j_{k+2}-j_{k+1})/(i_{k+2}-i_{k+1})
\label{jjii1}\\
&<&
(1+\Dep/\Tep)^2(j_{k+1}-j_k)/(i_{k+1}-i_k) \label{jjii2}\\
&<& (1+\Dep/\Tep)^2(1-\Dep/\Tep)^{-1}(j_{k+1}-j)/(i_{k+1}-i_k) \label{jjii3}\\
&<&
(1+4\Dep/\Tep)(j_{k+1}-j)/(i_{k+1}-i_k).\label{jjii4}
\end{eqnarray}
Here \eqref{jjii1} and \eqref{jjii3} are given by
\eqref{i1} 
(via
$j'-j_{k+1} < j_{k+2}-j_{k+1} +\Dep <
(1+\Dep/\Tep)(j_{k+2}-j_{k+1})$
and 
$j_{k+1}-j> j_{k+1}-j_k-\Dep > (1-\Dep/\Tep) (j_{k+1}-j_k)$);
\eqref{jjii2} is the assumption that \eqref{i2} fails; and \eqref{jjii4} holds if $\Tep > 5\Dep$.

Now assuming $\Tep/\Bep >4K$ ($K$ as in Lemma~\ref{BL2}),
Observation~\ref{PQobs} and Lemma~\ref{BL2} 
give
$\pr(\Psi_{i_{k+1},j_{k+1}}\mid \Psi_{i_1,j},\Psi_{i_{k+2},j'})\le \eps/2$ 
for any 
$j$ and $j'$ as above, and combining with \eqref{pryjB} gives 
the contradiction
\[
\pr(\Psi_{i_{k+1},j_{k+1}})<(1-\eps/2)\eps/2+\eps/2 < \eps,
\]
proving \eqref{i2}.

\mn
\emph{Proof of \eqref{i3}.} If $r_k\ge 3/\eps$, then for any $j=j_k\pm \Bep$, we have
$i_{k+1}-i_k<(\eps/2)(j_{k+1}-j)$ (using, say, $j_{k+1}-j_k>i_{k+1}-i_k>\Cep>3\Bep$).
So Observation~\ref{PQobs} and Lemma~\ref{BL1}(a) give
$\pr(\Psi_{i_{k+1},j_{k+1}}\mid \Psi_{i_k,j})\le \eps/2$,
and combining this with \eqref{pryjB} gives the desired contradiction.

\mn

This completes the proofs of Theorems~\ref{main-lemma} and \ref{pi-var}.
\end{proof}

\section{Proof of Theorem~\ref{down}}
\label{SecYd}

This is easy and we aim to be brief, omitting some routine verifications.

We first observe that (a) 
is best possible (for any $d$), while (b) is tight up to the $O_d(1)$:
for (a), let $P=\{(x,t):\mbox{$ x\in \{(1,0\dots 0),(0,1\dots 1)\}
\sub \mbN^{d-1}$, 
$t\leq |P|/2$}\}$; and for (b), 
let $P=\{te_i:i\in [d], t\leq \ell\}$, with $e_i$'s the standard basis vectors 
(so $\ell = (|P|-1)/d +1$).

\mn
\emph{Proof of Theorem~\ref{down}.}
{\rm (a)}
Since $P$ is not a chain, we may choose a direction $i$ with $|\{x_i:i\in P\}|> 1$,
and $x,y\in P$ 
with $x$ minimal subject to $x_i=\max\{z_i:z\in P\}$ 
and $y$ maximal subject to $y_i=\min\{z_i:z\in P\}$;
then $|\{z\in P:x,y\sim z\}|\leq 2$, implying $\max\{\pi(x),\pi(y)\}|\ge (|P|-2)/2$.

\nin
{\rm (b)}
We prove this with $O_d(1)=d^d$, which should be far from optimal.

Set $\ov{\pi}(x) = |P\sm \Pi(x)|$ and $|P|=n$.
For $i\in [d]$, let $\ell_i=\max\{t:te_i\in P\}$ and $m_i=|\{x\in P: x_i=\ell_i\}|$.
Then $\ell_im_i\leq n$ (since $P$ is an ideal) and $\ov{\pi}(\ell_ie_i) =\ell_i+m_i-1$.
So the theorem holds if there is an $i$ with $d\leq \ell_i\leq n/d$, 
and we assume there is not.
Call $i$ \emph{long} if $\ell_i> n/d$ and \emph{short} if $\ell_i<d$.
We may assume there is at least one $i$ of each type, since all long
would imply $|P|>n$, 
and if all are short then $n< d^d$.

Let $S$ be the set of short $i$'s, say (w.l.o.g.) $S=[s]$.
Let $\vp$ be the natural projection of $P$ on $\mbN^S$ and $B=\vp(P)\sub [d]^S$,
and for $y\in B$, set $f(y) =|\vp^{-1}(y)|$.
Then $f$ is nonincreasing (again, since $P$ is an ideal) and $f(\underline{1})\ge (d-s)n/d$
(since $\vp^{-1}(\underline{1})$ contains the disjoint intervals $[2e_i,\ell_ie_i]$, $i\in [s+1,d]$).
Choose $y$ maximal in $ B$ with $f(y)=\min\{f(z):z\in B\}$ (possible since $f$ is nonincreasing).
Then 
\[
f(y) < [n-(d-s)n/d]/(|B|-1)\leq (sn/d)/s=n/d
\]
(note $|B|\geq s+1$ since no coordinate hyperplane contains $P$),
and, for $x=(y,0\dots 0)$ ($\in P$), maximality of $y$ gives
\[
\ov{\pi}(x) = f(y) + |\{z\in B: z<y\}|  < n/d +d^d.
\]\qed

\end{document}